\documentclass[12pt]{article}
\usepackage{amsmath,amssymb,amsthm}
\usepackage{vmargin}
\usepackage{url}
\usepackage{nicefrac}

\usepackage{graphicx}

\setpapersize{USletter}
\setmargnohfrb{1.0in}{1.0in}{1.0in}{1.0in}

\AtBeginDocument{
\footskip=30pt
}

\usepackage{pxfonts}

\newtheorem{theorem}{Theorem}
\newtheorem{lemma}[theorem]{Lemma}
\newtheorem{corollary}[theorem]{Corollary}

\newcommand{\beq}{\begin{equation}}
\newcommand{\eeq}{\end{equation}}
\newcommand{\bea}{\begin{array}}
\newcommand{\eea}{\end{array}}

\title{On convergence of the Flint Hills series}
\author{Max A. Alekseyev\thanks{Department of Computer Science and Engineering,
University of South Carolina, Columbia, SC, U.S.A.\newline{Email: maxal@cse.sc.edu}} }

\begin{document}
\maketitle

\begin{abstract}
It is not known whether the Flint Hills series $\sum_{n=1}^{\infty} \frac{1}{n^3\cdot\sin(n)^2}$ converges.
We show that this question is closely related to the irrationality measure of $\pi$, denoted $\mu(\pi)$. 
In particular, convergence of the Flint Hills series would imply $\mu(\pi) \leq 2.5$ which is much stronger than the best currently known upper bound $\mu(\pi)\leq 7.6063\ldots$.

This result easily generalizes to series of the form $\sum_{n=1}^{\infty} \frac{1}{n^u\cdot |\sin(n)|^v}$ where $u,v>0$. 
We use the currently known bound for $\mu(\pi)$ to derive conditions on $u$ and $v$ that guarantee convergence of such series.
\end{abstract}

\section{Introduction}

Pickover~\cite{Pickover2002} defined the \emph{Flint Hills series} as $\sum_{n=1}^{\infty} \frac{1}{n^3\cdot\sin(n)^2}$ (named after Flint Hills, Kansas) and questioned whether it converges.
It was noticed that behavior of the partial sums of this series is closely connected to the rational approximations to $\pi$. 
In this paper we give a formal description of this connection, proving that
convergence of the Flint Hills series would imply an upper bound $2.5$ for the irrationality measure of $\pi$ which is much stronger 
than the best currently known bound $7.6063\ldots$ obtained by Salikhov~\cite{Salikhov2008}.
A rather slow progress in evaluating the irrationality measure of $\pi$ over past decades~\cite{Mahler1953,Mignotte1974,Chudnovsky1982,Hata1990,Hata1993a,Hata1993b,Salikhov2008} indicates the hardness of
this problem and suggests that the question of the Flint Hills series' convergence would unlikely be resolved in the nearest future.

The \emph{irrationality measure} $\mu(x)$ of a positive real number $x$ is defined as the infimum of such $m$ that the inequality
$$ 0 < \left|x - \frac{p}{q}\right| < \frac{1}{q^m}$$
holds only for a finite number of co-prime positive integers $p$ and $q$. If no such $m$ exists, then $\mu(x) = +\infty$ (in which case $x$ is called \emph{Liouville number}).

Informally speaking, the larger is $\mu(x)$, the better $x$ is approximated by rational numbers.
It is known that $\mu(x)=1$ if $x$ is a rational number; $\mu(x)=2$ if $x$ is irrational algebraic number (Roth's theorem~\cite{Roth1955} for which Roth was awarded the Fields Medal); 
and $\mu(x)\geq 2$ if $x$ is a transcendental number. Proving that $\mu(x) > 1$ is a traditional way to establish irrationality of $x$, with the most remarkable example 
of the $\zeta(3)$ irrationality (where $\zeta(s)=\sum_{n=1}^{\infty} n^{-s}$ is the Riemann zeta function) proved by Apery~\cite{Apery1979,Poorten1979}.


\section{Convergence of the Flint Hills series}

\begin{lemma}\label{Lsinx} For a real number $x$, we have
$$|\sin(x)| \leq |x|.$$
Furthermore, if $|x|\leq\nicefrac{\pi}{2}$ then
$$|\sin(x)| \geq \frac{2}{\pi}\cdot |x|.$$

\end{lemma}

\begin{proof} The former bound follows from the integral estimate
$$|\sin(x)| = \left| \int_0^x \cos y\cdot\mathrm{d}y \right| \leq \int_0^{|x|} | \cos y |\cdot\mathrm{d}y \leq \int_0^{|x|} 1\cdot\mathrm{d}y = |x|,$$

To prove the latter bound, we notice that $|\sin(x)|=\sin(|x|)$ and without loss of generality assume that $0\leq x\leq \nicefrac{\pi}{2}$. 
Let $x_0 = \arccos(\nicefrac{2}{\pi})$ so that for $x\leq x_0$ we have $\cos(x)\geq \nicefrac{2}{\pi}$ and thus
$$\sin(x) = \int_0^x \cos(y)\cdot\mathrm{d}y \geq \int_0^x \frac{2}{\pi}\cdot\mathrm{d}y = \frac{2}{\pi}\cdot x,$$
while for $x\geq x_0$ we have $\cos(x)\leq \nicefrac{2}{\pi}$ and thus
$$\sin(x) = 1 - \int_x^{\nicefrac{\pi}2} \cos(y)\cdot\mathrm{d}y \geq 1 - \int_x^{\nicefrac{\pi}2} \frac{2}{\pi}\cdot\mathrm{d}y 
= 1 - \frac{2}{\pi}\cdot\left(\frac{\pi}{2} - x\right) = \frac{2}{\pi}\cdot x.$$
\end{proof}

\begin{theorem}\label{Tmain}
For positive real numbers $u$ and $v$, 
$\frac{1}{n^u\cdot |\sin(n)|^v} = O\left(\frac{1}{n^{u - (\mu(\pi)-1)\cdot v - \epsilon}}\right)$ for any $\epsilon>0$.
Furthermore,
\begin{enumerate}
\item If $\mu(\pi)< 1+\nicefrac{u}{v}$, the sequence $\frac{1}{n^u\cdot |\sin(n)|^v}$ converges (to zero);

\item If $\mu(\pi) > 1+\nicefrac{u}{v}$, the sequence $\frac{1}{n^u\cdot |\sin(n)|^v}$ diverges.
\end{enumerate}
\end{theorem}


\begin{proof} Let $\epsilon>0$ and $k=\mu(\pi)+\nicefrac{\epsilon}{v}$. Then
the inequality
\begin{equation}\label{Epiapprox}
\left|\pi - \frac{p}{q}\right| < \frac{1}{q^k}
\end{equation}
holds only for a finite number of co-prime positive integers $p$ and $q$. 

For a positive integer $n$, let $m = \left\lfloor \nicefrac{n}{\pi} \right\rfloor$ so that $\left|\nicefrac{n}{\pi}-m\right|\leq \nicefrac{1}{2}$ and thus 
$\left|n-m\cdot\pi\right|\leq \nicefrac{\pi}{2}$. Then by Lemma~\ref{Lsinx},
$$|\sin(n)| = |\sin(n-m\cdot \pi)| \geq \frac{2}{\pi}\cdot |n-m\cdot \pi| = \frac{2}{\pi}\cdot m\cdot\left|\frac{n}{m}-\pi\right|.$$

On the other hand, for large enough $n$ and $m$, we have $|\nicefrac{n}{m}-\pi|\geq \nicefrac{1}{m^k}$, implying that
$$|\sin(n)| \geq \frac{2}{\pi}\cdot m\cdot\left|\frac{n}{m}-\pi\right| \geq \frac{2}{\pi}\cdot \frac{1}{m^{k-1}} \geq c\cdot \frac{1}{n^{k-1}}$$
for some constant $c>0$ depending only on $k$ but not $n$ (since $\nicefrac{n}{m}$ tends to $\pi$ as $n$ grows).

Therefore, for all large enough $n$, we have 
$$\frac{1}{n^u\cdot |\sin(n)|^v} \leq \frac{1}{c^v\cdot n^{u-(k-1)\cdot v}} = O\left(\frac{1}{n^{u - (\mu(\pi)-1)\cdot v - \epsilon}}\right).$$

The statement 1 now follows easily. If $\mu(\pi)<1+\nicefrac{u}{v}$, we take $\epsilon = \nicefrac{v}{2}\cdot (1+\nicefrac{u}{v} - \mu(\pi))$ to obtain
$$\frac{1}{n^u\cdot |\sin(n)|^v} = O\left(\frac{1}{n^{u - v\cdot (\mu(\pi)-1) - \epsilon}}\right) = O\left(\frac{1}{n^{\epsilon}}\right).$$


Now let us prove statement 2. If $\mu(\pi) > 1+\nicefrac{u}{v}$, then for $k=1+\nicefrac{u}{v}$ the inequality \eqref{Epiapprox} 
holds for infinitely many co-prime positive integers $p$ and $q$. That is, there exists a sequence of rationals $\nicefrac{p_i}{q_i}$ such that
$\left|p_i - \pi\cdot q_i\right| < \frac{1}{q_i^{k-1}}$. Then
$$|\sin(p_i)| = |\sin(p_i-q_i\cdot \pi)| \leq |p_i-q_i\cdot \pi| < \frac{1}{q_i^{k-1}} < C\cdot \frac{1}{p_i^{k-1}}$$
for some constant $C>0$ depending only on $k$.

Therefore, for $n=p_i$ we have
$$\frac{1}{n^u\cdot |\sin(n)|^v} > C^v\cdot n^{v\cdot(k-1)-u} = C^v.$$
On the other hand, we have
$$|\sin(1+p_i)| = |\sin(1+p_i-q_i\cdot \pi)|\quad \mathop{\longrightarrow}\limits_{i\to\infty}\quad \sin(1)$$
and thus
$$\frac{1}{(1+p_i)^u\cdot |\sin(1+p_i)|^v}\quad \mathop{\longrightarrow}\limits_{i\to\infty}\quad 0.$$
We conclude that the sequence $\frac{1}{n^u\cdot |\sin(n)|^v}$ diverges, since it contains two subsequences one which is bounded from below by a positive constant, while the other tends to zero.

\end{proof}


\begin{corollary}\label{Cor1} For positive real numbers $u$ and $v$,
\begin{enumerate} 
\item If the sequence $\frac{1}{n^u\cdot |\sin(n)|^v}$ converges, then $\mu(\pi)\leq 1+\nicefrac{u}{v}$;

\item If the sequence $\frac{1}{n^u\cdot |\sin(n)|^v}$ diverges, then $\mu(\pi) \geq 1+\nicefrac{u}{v}$.
\end{enumerate}
\end{corollary}

\begin{corollary} 
If the Flint Hills series $\sum_{n=1}^{\infty} \frac{1}{n^3\cdot \sin(n)^2}$ converges, then $\mu(\pi) \leq \nicefrac{5}{2}$.
\end{corollary}

\begin{proof}
Convergence of $\sum_{n=1}^{\infty} \frac{1}{n^3\cdot \sin(n)^2}$ implies that $\lim\limits_{n\to\infty} \frac{1}{n^3\cdot \sin(n)^2} = 0$ 
and thus by Corollary~\ref{Cor1}, $\mu(\pi)\leq\nicefrac{5}{2}$.
\end{proof}

\begin{theorem}\label{Tsum}
For positive real numbers $u$ and $v$,
if $\mu(\pi)< 1+\nicefrac{(u-1)}{v}$, then $\sum_{n=1}^{\infty} \frac{1}{n^u\cdot |\sin(n)|^v}$ converges.
\end{theorem}

\begin{proof} The inequality $\mu(\pi)< 1+\nicefrac{(u-1)}{v}$ implies that $u-v\cdot(\mu(\pi)-1)>1$. Then there exists $\epsilon>0$ such that $w=u-v\cdot(\mu(\pi)-1)-\epsilon>1$.
By Theorem~\ref{Tmain}, $\frac{1}{n^u\cdot |\sin(n)|^v} = O\left(\frac{1}{n^w}\right)$ further implying that
$$\sum_{n=1}^{\infty} \frac{1}{n^u\cdot |\sin(n)|^v} = O\left(\zeta(w)\right) = O(1).$$
\end{proof}

\begin{corollary}\label{Csercon}
For positive real numbers $u$ and $v$,
if $\sum_{n=1}^{\infty} \frac{1}{n^u\cdot |\sin(n)|^v}$ diverges, then $\mu(\pi) \geq 1+\nicefrac{(u-1)}{v}$.
\end{corollary}

Unfortunately, the divergence of the Flint Hills series would not imply any non-trivial result per Corollary~\ref{Csercon}.

\section{Known bounds for $\mu(\pi)$ and their implications}

Since $\pi$ is a transcendental number, $\mu(\pi)\geq 2$. To the best of our knowledge, no better lower bound for $\mu(\pi)$ is currently known.

The upper bound for $\mu(\pi)$ has been improved over the past decades. Starting with the bound $\mu(\pi)\leq 30$ established by Mahler in 1953~\cite{Mahler1953}, it was improved to $\mu(\pi)\leq 20$ by Mignotte in 1974~\cite{Mignotte1974}, 
and then to $\mu(\pi)\leq 19.8899944\ldots$ by Chudnovsky in 1982~\cite{Chudnovsky1982}. In 1990-1993 Hata in a series of papers~\cite{Hata1990,Hata1993a,Hata1993b} 
decreased the upper bound down to $\mu(\pi)\leq 8.016045\ldots$. The best currently known upper bound $\mu(\pi)\leq 7.6063\ldots$ was obtained in 2008 by Salikhov~\cite{Salikhov2008}. 

By Theorem~\ref{Tmain}, the Salikhov's bound implies that the sequence $\frac{1}{n^u\cdot |\sin(n)|^v}$ converges to zero as soon as
$1+\nicefrac{u}{v} > 7.6063$, including in particular the pairs $(u,v) = (7,1)$, $(14,2)$, $(20,3)$ etc. Correspondingly, Theorem~\ref{Tsum} further implies 
that the series $\sum_{n=1}^{\infty} \frac{1}{n^u\cdot |\sin(n)|^v}$ converges for $(u,v) = (8,1)$, $(15,2)$, $(21,3)$ etc.


\begin{footnotesize}
\bibliographystyle{plain} 
\bibliography{flint.bib} 

\newcommand{\noopsort}[1]{}
\begin{thebibliography}{10}

\bibitem{Apery1979}
R.~Apery.
\newblock Irrationalite de $\zeta(2)$ et $\zeta(3)$.
\newblock {\em Asterisque}, 61:11--13, 1979.

\bibitem{Chudnovsky1982}
G.~V. Chudnovsky.
\newblock {Hermite-pade approximations to exponential functions and elementary
  estimates of the measure of irrationality of $\pi$}.
\newblock {\em Lecture Notes in Math.}, 925:299--322, 1982.

\bibitem{Hata1990}
M.~Hata.
\newblock Legendre type polynomials and irrationality measures.
\newblock {\em J. Reine Angew. Math.}, 407:99--125, 1990.

\bibitem{Hata1993a}
M.~Hata.
\newblock A lower bound for rational approximations to $\pi$.
\newblock {\em J. Number Theor.}, 43(1):51--67, 1993.

\bibitem{Hata1993b}
M.~Hata.
\newblock Rational approximations to $\pi$ and some other numbers.
\newblock {\em Acta Arith.}, 63(4):335--349, 1993.

\bibitem{Mahler1953}
K.~Mahler.
\newblock On the approximation of $\pi$.
\newblock {\em Nederl. Akad. Wetensch. Proc. Ser. A}, 56:30--42, 1953.

\bibitem{Mignotte1974}
M.~Mignotte.
\newblock Approximations rationnelles de $\pi$ et quelques autres nombres.
\newblock {\em Bull. Soc. Math. Fr., Suppl.}, 37:121--132, 1974.

\bibitem{Pickover2002}
C.~A. Pickover.
\newblock {\em {The Mathematics of Oz: Mental Gymnastics from Beyond the
  Edge}}, volume~2.
\newblock {Cambridge University Press}, 2002.
\newblock Chapter 25 ``Flint Hills Series''.

\bibitem{Roth1955}
K.~F. Roth.
\newblock Rational approximations to algebraic numbers.
\newblock {\em Mathematika}, 2:1--20, 168, 1955.

\bibitem{Salikhov2008}
V.~Kh. Salikhov.
\newblock {On the Irrationality Measure of pi}.
\newblock {\em {Russ. Math. Surv}}, 63(3):570--572, 2008.

\bibitem{Poorten1979}
A.~van~der Poorten.
\newblock A proof that euler missed...
\newblock {\em The Mathematical Intelligencer}, 1(4):195--203, 1979.

\end{thebibliography}
\end{footnotesize}

\end{document}